\documentclass[oneside,a4paper]{article}
\usepackage{amssymb,amsmath,latexsym,graphicx,tikz,indentfirst,scalerel,hyperref,caption,marginnote,multicol,cite}

\usepackage[raggedright]{titlesec}

\usepackage[english]{babel}
\usepackage{csquotes}

\hypersetup{
  colorlinks   = true, 
  urlcolor     = blue, 
  linkcolor    = blue, 
  citecolor   = red 
}

\usetikzlibrary{automata, positioning, calc,arrows}

\def\dj{d\kern-0.4em\char"16\kern-0.1em}
\def\Dj{\mbox{\raise0.3ex\hbox{-}\kern-0.4em D}}

\newtheorem{theorem}{Theorem}[section]
\newtheorem{lemma}[theorem]{Lemma}

\newtheorem{problem}[theorem]{Problem}

\makeatletter
\renewcommand{\@dotsep}{10000}
\makeatother

\newenvironment{proof}
{\noindent
{\it Proof.}}
{\hspace{\stretch{1}}%
$\Box$}

\newcounter{primer}[section]

\DeclareMathOperator{\Aut}{Aut}

\DeclareMathOperator{\Z}{Z}

\makeatletter 
\tikzset{my loop/.style =  {to path={
  \pgfextra{}
  [looseness=4,min distance=5mm]
  \tikz@to@curve@path},font=\sffamily\small
  }}  
\makeatletter

\makeatletter
\newcommand{\myitem}[1]{%
\item[#1]\protected@edef\@currentlabel{#1}%
}
\makeatother

\begin{document}

\thispagestyle{empty}
\begin{center}
\Large{Power graphs of all nilpotent groups}
\vspace{3mm}

\normalsize{S.H. Jafari, S. Zahirovi\' c}
\end{center}

\begin{abstract}
The directed power graph $\vec{\mathcal G}(\mathbf G)$ of a group $\mathbf G$ is the simple digraph with vertex set $G$ such that $x\rightarrow y$ if $y$ is a power of $x$. The power graph $\mathcal G(\mathbf G)$ of the group $\mathbf G$ is the underlying simple graph.

In this paper, we prove that Pr\" ufer group is the only nilpotent group whose power graph does not determine the directed power graph up to isomorphism. Also, we present a group $\mathbf G$ with quasicyclic torsion subgroup that is determined by its power graph up to isomorphism, i.e. such that $\mathcal G(\mathbf H)\cong\mathcal G(\mathbf G)$ implies $\mathbf H\cong \mathbf G$ for any group $\mathbf H$.
\end{abstract}

\section{Introduction} 

The directed power graph $\vec{\mathcal G}(\mathbf G)$ of a group $\mathbf G$ is the simple directed graph with vertex set $G$ in which $x\rightarrow y$ if $y$ is a power of $x$. The power graph $\mathcal G(\mathbf G)$ of a group $\mathbf G$ is the simple graph whose vertex set is $G$ and whose vertices $x$ and $y$ are adjacent if one of them is a power of the other. The directed power graph was introduced by Kelarev and Quinn \cite{kelarev-quinn-1}, and the power graph was first studied by Chakrabarty, Ghosh and Sen \cite{cakrabarti}. Many authors have actively studied the power graph recently (see e.g. \cite{cameron-ghosh, cameron, cameron-guerra-jurina, aalipour-et-al, shitov, kelarev-quinn-2, kelarev-3, kelarev-4, cameron-jafari, zahirovic-1, zahirovic-2,cameron-manna-mehatari, ma-feng-wang-1, doostabadi-ghouchan, jafari-1, jafari-2, jafari-chattopadhyay, pourghobadi-jafari}). The reader may refer to survey papers \cite{survey-kelarev} and \cite{survey-cameron}, as well as \cite{cameron-graphs-defined-on-groups}, which concerns various graphs defined on groups and comparisons between them.

Many combinatorial and algebraic properties of the power graph were studied. Aalipour et al. \cite{aalipour-et-al} proved that every group of finite exponent has perfect power graph. Shitov \cite{shitov} proved that the chromatic number of every power-associate groupoid is at most countable. Cameron and Jafari \cite{cameron-jafari} characterized all groups whose power graphs have finite independence numbers, and they showed that the power graph of every such group has equal independence number and clique cover number. In \cite{cameron-manna-mehatari}, the authors determined all groups whose power graphs are threshold or split graphs.

The following question regarding the power graph has also received significant attention: for which groups does the power graph determine the directed power graph. Cameron and Ghosh \cite{cameron-ghosh} proved that, given that finite groups $\mathbf G$ and $\mathbf H$ are abelian, $\mathcal G(\mathbf G)\cong\mathcal G(\mathbf H)$ implies $\mathbf G\cong\mathbf H$, and therefore, it also implies that $\vec{\mathcal G}(\mathbf G)\cong\vec{\mathcal G}(\mathbf H)$. In a subsequent paper, Cameron \cite{cameron} showed that $\mathcal G(\mathbf G)\cong\mathcal G(\mathbf H)$ implies $\vec{\mathcal G}(\mathbf G)\cong\vec{\mathcal G}(\mathbf H)$ for any two finite groups $\mathbf G$ and $\mathbf H$. Since the power graph of a Pr\" ufer group is a complete graph, it is not true that the power graph of a group always determines the directed power graph. However, all groups whose power graphs determine their directed power graphs have not been classified. In \cite{cameron-guerra-jurina}, the authors have shown that, if both $\mathbf G$ and $\mathbf H$ are torsion-free groups of nilpotency class $2$, then $\mathcal G(\mathbf G)\cong\mathcal G(\mathbf H)$ implies $\vec{\mathcal G}(\mathbf G)\cong\vec{\mathcal G}(\mathbf H)$. In later research on this topic, quasicyclic subgroups have played an important role. We say that a quasicyclic subgroup $\mathbf C_{p^\infty}$ of a group $\mathbf G$ is intersection-free if $\mathbf C_{p^\infty}\cap\mathbf C$ is trivial for all cyclic subgroups $\mathbf C\leq\mathbf G$ such that $\mathbf C\not\leq\mathbf C_{p^\infty}$. Zahirovi\' c \cite{zahirovic-1,zahirovic-2} showed that, for group $\mathbf G$ and $\mathbf H$ such that $\mathbf G$ does not have an intersection-free quasicyclic subgroup, $\mathcal G(\mathbf G)\cong\mathcal G(\mathbf H)$ implies $\vec{\mathcal G}(\mathbf G)\cong\vec{\mathcal G}(\mathbf H)$. Therefore, the power graph of any torsion-free groups determines the directed power graph. In this paper, we prove that, given that both $\mathbf G$ and $\mathbf H$ are non-quasicyclic nilpotent groups, $\mathcal G(\mathbf G)\cong\mathcal G(\mathbf H)$ implies $\vec{\mathcal G}(\mathbf G)\cong\vec{\mathcal G}(\mathbf H)$. Further, for a prime number $p$ and any group $\mathbf G$, we show that $\mathcal G(\mathbf G)\cong\mathcal G(\mathbb Q\times\mathbf C_{p^\infty})$ only if $\mathbf G\cong \mathbb Q\times\mathbf C_{p^\infty}$. Therefore, $\mathbb Q\times\mathbf C_{p^\infty}$ is a group with an intersection-free quasicyclic subgroup whose power graph determines the directed power graph. Thus, we show in this paper that the sufficient condition from \cite[Theorem 21]{zahirovic-2} is not a necessary one as well.

As stated above, by the main result of this paper, for nilpotent groups $\mathbf G$ and $\mathbf H$ we have the following.
\[\begin{array}{r}
       \mathcal G(\mathbf G)\cong\mathcal G(\mathbf H) \\
       \mathbf G\text{ is not a Pr\" ufer group}
        \end{array}\Bigg\} \Rightarrow\vec{\mathcal G}(\mathbf G)\cong\vec{\mathcal G}(\mathbf H)\]
Naturally, we ask ourselves whether the same implication holds even without the assumption that both groups $\mathbf G$ and $\mathbf H$ are nilpotent. Therefore, we mention the following problem.

\begin{problem}
Let $\mathbf G$ be a non-quasicyclic nilpotent group, and let $\mathbf H$ be any group such that $\mathcal G(\mathbf H)\cong\mathcal G(\mathbf G)$. Do $\mathbf G$ and $\mathbf H$ have isomorhic directed power graphs as well?
\end{problem}

\section{Basic notions and notations}

In this paper, graphs are denoted by capital letters of Greek alphabet. By $\Gamma(V,E)$, or shortly by $\Gamma$, we denote the graph whose vertex set and edge set are $V$ and $E$, respectively. A graph $\Delta(V_2,E_2)$ is a subgraph of $\Gamma(V,E)$ if $V_2\subseteq V$ and $E_2\subseteq E$. Furthermore, $\Delta$ is an induced subgraph of $\Gamma$ if $\Delta$ contains no pair of nonadjacent vertices that are adjacent in $\Gamma$; we also say that that $\Delta$ is a subgraph of $\Gamma$ induced by $V_2$. The induced subgraph of $\Gamma$ by a set of vertices $W$ is denoted by $\Gamma[W]$. For vertices $x$ and $y$ of $\Gamma$, $x\sim y$ denotes that $x$ and $y$ are adjacent in $\Gamma$. Also, for a digraph $\vec\Gamma$ and its vertices $x$ and $y$, we denote by $x\rightarrow y$ that there is an arc from $x$ to $y$. A clique is a set of vertices which induces a complete subgraph.

Throughout this paper, algebraic structures are denoted by bold capital letters, and we denote the universe of an algebra by the respective roman capital letter. We say that $x$ is an $n${\bf -th root} of an element $y$ of a group if $y=x^n$. If $n$ is a prime number, then we say that $x$ is a {\bf prime root} of $y$. If $P=\prod_{i=1}^n A_i$ for some sets $A_i$, then $\pi_i$ denotes the respective projection map, i.e. $\pi_i(a_1,a_2,\dots,a_n)=a_i$. Also, for any $S\subseteq P$, $\pi_i(S)=\{\pi(\overline a)\mid \overline a\in S\}$. We say that a quasicyclic subgroup $\mathbf C_{p^\infty}$ of $\mathbf G$ is {\bf intersection-free} if, for every cyclic subgroup $\mathbf C$ of $\mathbf G$, $\mathbf C\cap\mathbf C_{p^\infty}$ is nontrivial only if $\mathbf C\leq\mathbf C_{p^\infty}$.

Let $\mathbf G$ be a group. The {\bf directed power graph} of $\mathbf G$ is the graph $\vec{\mathcal G}(\mathbf G)$ whose vertex set is $G$ and such that, for any pair of elements $x,y\in G$, there is an arc from $x$ to $y$ if there is nonzero integer $n$ such that $y^n=x$. By $x\rightarrow y$, we denote that there is an arc from $x$ to $y$ in the directed power graph of $\mathbf G$. The {\bf power graph} $\mathcal G(\mathbf G)$ of the group $\mathbf G$ is defined as the underlying simple graph of $\vec{\mathcal G}(\mathbf G)$. For elements $x$ and $y$ of $\mathbf G$, $x\sim y$ denotes that $x$ and $y$ are adjacent in $\mathcal G(\mathbf G)$.

The closed neighborhood $\overline N(x)$ of a vertex $x$ of a graph $\Gamma$ is the set which included $x$ and all vertices of $\Gamma$ adjacent to $x$. For a group $\mathbf G$ and for $x,y\in G$, we write $x\equiv y$ if $\overline N(x)=\overline N(y)$ in $\mathcal G(\mathbf G)$. Also, for an element $x$ of a group $\mathbf G$, $O(x)$ denotes the set $\big\{x^n\mid n\in\mathbb Z\setminus\{-1,0,1\}\big\}$.

In this paper, for a group $\mathbf G$, $G_{<\infty}$ and $G_\infty$ denote the set of all finite-order elements of $\mathbf G$ and the set of all infinite-order elements of $\mathbf G$, respectively. The {\bf finite-order segment} of $\mathcal G(\mathbf G)$ is the subgraph of $\mathcal G(\mathbf G)$ induced by $G_{<\infty}$, and the {\bf infinite-order segment} of $\mathcal G(\mathbf G)$ is the subgraph of $\mathcal G(\mathbf G)$ induced $G_{\infty}$. The finite-order segment and the infinite-order segment of the directed power graph of a group are defined similarly. Since no element of finite order is adjacent to an element of infinite order in $\mathcal G(\mathbf G)$, and since the identity element is adjacent to all elements of finite order, the finite-order segment is a connected component of $\mathcal G(\mathbf G)$. The infinite-order segment, however, may not be a connected component of the power graph of a group, but it is a union of connected components.

\section{The power graph and the directed power graph of nilpotent group}\label{power graph determines the directed power graph for almost all abelian groups}

In this section, we prove that any pair of non-quasicyclic nilpotent groups with isomorphic power graphs have isomorphic directed power graphs, too. We shall need the following well-known result on nilpotent groups (see \cite[5.2.7]{robinson}).

\begin{theorem}\label{about all nilpotent groups}
Let $\mathbf G$ be a nilpotent group. Then $\mathbf G_{<\infty}$ is a normal subgroup of $\mathbf G$ such that $\mathbf G/\mathbf G_{<\infty}$ is torsion-free, and $\mathbf G_{<\infty}$ is the direct sum of the unique maximal $p$-subgroups of $\mathbf G$.
\end{theorem}

Let us show first that, if the torsion subgroup of a nilpotent group is quasicyclic, then it is a central subgroup.

\begin{lemma}\label{all finite is central}
Let $\mathbf G$ be a nilpotent group, and let $\mathbf G_{<\infty}\cong\mathbf C_{p^\infty}$ for some prime number $p$. Then $\mathbf G_{<\infty}$ is a central subgroup of $\mathbf G$. 
\end{lemma}

\begin{proof}
Let $a\in G_\infty$. Let us prove that $\mathbf K=\langle a, G_{<\infty}\rangle$ is abelian. Since $\mathbf G_{<\infty}\unlhd\mathbf G$, $\mathbf K=\langle a\rangle \mathbf G_{<\infty}$. Moreover, $\mathbf K=\langle a\rangle \ltimes\mathbf G_{<\infty}$. Since $\mathbf G$ is nilpotent, $\mathbf K$ is nilpotent as well.

Suppose now that $\mathbf K$ is not abelian, and let $\mathbf K_2=[\mathbf K,\mathbf K]$ and $\mathbf K_3=[\mathbf K_2,\mathbf K]$. Then $\mathbf K> \mathbf K_2> \mathbf K_3$, $\mathbf K_2$ is nontrivial, and $\mathbf K/\mathbf K_3$ is nonabelian. Let us denote $\mathbf K/\mathbf K_3$ by $\mathbf L$. Notice that $\mathbf L=\overline{\langle a\rangle}\ltimes\overline{\mathbf G_{<\infty}}$, where $\overline{\mathbf H}$ denotes $(\mathbf H\mathbf K_3)/\mathbf K_3$ for any $\mathbf H\leq\mathbf K$. Also, since $\mathbf L/\overline{\mathbf G_{<\infty}}$ is abelian, we have that $\mathbf L'\leq\overline{\mathbf G_{<\infty}}$. Furthermore, since $\mathbf L$ is nonabelian, $\overline{\mathbf G_{<\infty}}$ is nontrivial, and thus, $\mathbf G_{<\infty}\not\leq\mathbf K_3$. Therefore, by the second isomorphism theorem,
\[\overline{\mathbf G_{<\infty}}=\frac{\mathbf G_{<\infty}\mathbf K_3}{\mathbf K_3}\cong \frac{\mathbf G_{<\infty}}{\mathbf G_{<\infty}\cap\mathbf K_3}\neq 1.\]
Therefore, $\overline{\mathbf G_{<\infty}}\cong \mathbf C_{p^\infty}$. Note that $\mathbf L$ is nilpotent of class $2$. Also, we remind the reader of the commutator identity $[x,yz]=[x,z]z^{-1}[x,y]z$. Since $[\mathbf L,\mathbf L]$ is a central subgroup of $\mathbf L$, then $[x,yz]=[x,y][x,z]$, and therefore,
\[[x,y^n]=[x,y]^n\]
for every $x,y,z\in L$ and $n\in\mathbb N$, and similarly $[x^n,y]=[x,y]^n$. Therefore,
\begin{equation*}\label{the L'}
L'=\{ [b,y]\mid y\in \overline{G_{<\infty}} \},
\end{equation*}
where $b=aK_3$. Furthermore, $\mathbf L'$ is a $p$-group since
\[ [b,y]^{p^n}=[b,y^{p^n}]=[b,e]=e, \]
where $p^n$ is the order of $y\in\overline{G_{<\infty}}$. Let us show that $\mathbf L'$ is infinite, too. Since $\mathbf L$ is nonabelian, there is $y_0\in \overline{G_{<\infty}}$ such that $[b,y_0]\neq e$. Moreover, for every $n\in\mathbb N$, there is $y_n\in\overline{G_{<\infty}}$ such that $[b,y_n]^{p^n}=[b,y_n^{p^n}]=[b,y_0]$, i.e. $o(y_n)> p^n$. Therefore, $\mathbf L'$ is infinite, and $\mathbf L'\cong \mathbf C_{p^\infty}$. Now, since $\mathbf L=\overline{\langle a\rangle}\ltimes\overline{\mathbf G_{<\infty}}$ has the unique quasicyclic subgroup, it follows that $\overline{\mathbf G_{<\infty}}=\mathbf L'\leq\Z(\mathbf L)$, which is in contradiction with the assumption that $\mathbf L$ is not abelian. This proves the lemma.
\end{proof}

Now, we are ready to show the following two lemmas about numbers of prime roots of an infinite-order element. We will use those results in the proof of Lemma \ref{numbers of prime roots}, which is the key lemma for proving the main result of this section.

\begin{lemma}\label{p many p-th roots}
Let $\mathbf G$ be a nilpotent group, and let $\mathbf G_{<\infty}\cong\mathbf C_{p^\infty}$ for some prime number $p$. Then the number of $p$-th roots of any infinite-order element is either $p$ or $0$.
\end{lemma}

\begin{proof}
Let $a$ be a $p$-th root of an infinite-order element $u\in\mathbf G$. Let us prove first that $\langle a, G_{<\infty}\rangle=\langle a\rangle\mathbf G_{<\infty}$ contains all $p$-th roots of $u$. Suppose that $u$ has a $p$-th root $b\not\in\langle a \rangle\mathbf G_{<\infty}$, and let $\mathbf L$ denote $\langle a, b, G_{<\infty}\rangle$. Since $\mathbf L$ is nilpotent, it has the lower central series
\[ \mathbf L=\mathbf L_1>\mathbf  L_2 > \cdots >\mathbf L_t>\mathbf L_{t+1}=1, \]
where $\mathbf L_{i+1}=[\mathbf L_i,\mathbf L]$ for all $i$.

In order to show that $\langle a\rangle\mathbf G_{<\infty}$ contains all $p$-th roots of $u$, we need to show first that $\mathbf L$ has nilpotency class at most $2$. So suppose that $t>2$, and let $x\in L_{t-1}$. Then $[x,a],[x,b]\in L_t\leq\Z(\mathbf L)$. Therefore, $[x,yz]=[x,z]z^{-1}[x,y]z=[x,z][x,y]$ for any $y,z\in L$, and thus,
\begin{align*}
&[x,a]^p=[x,a^p]=e\text{ and}\\
&[x,b]^p=[x,b^p]=e.
\end{align*}
Furthermore, since $L=\langle a,b,G_{<\infty}\rangle$, $[x,y]^p=e$ for all $y\in L$, which implies that the exponent of $\mathbf L_t$ is $p$. Therefore, $\mathbf L_t$ is the unique subgroup of order $p$ of $\mathbf G_{<\infty}$. Now, let us consider $\mathbf L/\mathbf L_t$. It's lower central series is
\[ \mathbf L/\mathbf L_t>\mathbf  L_2/\mathbf L_t > \cdots > \mathbf L_{t-1}/\mathbf L_t>1, \]
and similarly, we conclude that $\mathbf L_{t-1}/\mathbf L_t>1$ is the unique subgroup of order $p$ of $ \mathbf L/\mathbf L_t$. Therefore, $\mathbf L_{t-1}$ is the unique subgroup of $\mathbf G_{<\infty}$ of order $p^2$. It follows that $\mathbf L_t=[\mathbf L_{t-1},\mathbf L]=1$, which is a contradiction. Thus, $\mathbf L$ has nilpotency class at most $2$.

Now, since
\[ b[b,a]a^{-1}=a^{-1}b, \]
and since $[a,b]\in \mathbf L'\leq\Z(\mathbf L)$, for some $t\in\mathbb N$, we obtain
\[ (a^{-1}b)^p[b,a]^t=(a^{-1})^pb^p=(a^p)^{-1}b^p=u^{-1}u=e. \]
Therefore, since $[b,a]\in G_{<\infty}$, it follows that $b\in a G_{<\infty}\leq\langle a\rangle\mathbf G_{<\infty}$, which is a contradiction. This proves that $\langle a\rangle\mathbf G_{<\infty}$ contains all $p$-th roots of $u$. Moreover, by Lemma \ref{all finite is central}, $\langle a\rangle\mathbf G_{<\infty}=\langle a\rangle\times\mathbf G_{<\infty}$, and therefore, u has $p$ many $p$-th roots in $\langle a\rangle\times \mathbf G_{<\infty}$, which proves the lemma. 
\end{proof}

\begin{lemma}\label{one q-th root}
Let $p$ and $q$ be different prime numbers, let $\mathbf G$ be a nilpotent group such that $\mathbf G_{<\infty}\cong\mathbf C_{p^\infty}$. Then the number of $q$-th roots of any infinite-order element is at most $1$.
\end{lemma}

\begin{proof}
In order to prove this lemma, suppose that $a^q=b^q=u$ for some $u\in G_\infty$ and $a,b\in G$. Let $\mathbf L$ denote $\langle a,b,G_{<\infty}\rangle$. Notice that $\mathbf L$ is nilpotent. In order to show that $\mathbf L$ is abelian, let us assume that $\mathbf L$ has nilpotency class at least $2$. Let
\[ \mathbf L=\mathbf L_1> \mathbf L_2>\cdots > \mathbf L_t> \mathbf L_{t+1}=1 \]
be its lower central series. Let $y\in L_{t-1}$. Then $[y,a],[y,b]\in \mathbf L_t\leq\Z(\mathbf L)$, and therefore, $[y,a]^q=[y,a^q]=1$ and $[y,b]^q=[y,b^q]=1$. Since $\mathbf G$ contains no elements of order $q$, it follows that $[y,a]=[y,b]=1$. Therefore, $y\in\Z(\mathbf L)$. Thus, $\mathbf L_{t-1}\leq\Z(\mathbf L)$, which is a contradiction. This proves that $\mathbf L$ is an abelian group. Now, we have
\[ (ab^{-1})^q=a^q(b^{-1})^q=uu^{-1}=e, \]
and since $\mathbf G$ contains no element of order $q$, we conclude that $a=b$. Thus, the lemma has been proven.
\end{proof}

\begin{lemma}\label{numbers of prime roots}
Let $\mathbf G$ be nilpotent group, let $\mathbf G_{<\infty}\cong\mathbf C_{p^\infty}$ for some prime number $p$, and let $u\in G_\infty$. If $u$ has a $p$-th root, then the minimum number of maximal locally cyclic subgroups of $\mathbf G$ whose union contains all prime roots of $u$ is $p$. Otherwise, that number is $1$.
\end{lemma}

\begin{proof}
Note that, for two different $p$-th roots of $u\in G$ of infinite order, there is no cyclic nor locally cyclic subgroup that contains both of them. Therefore, the minimum number of maximal locally cyclic subgroups of $\mathbf G$ whose union contains all $p$-th roots of $u$ is at most $p$. 

We shall prove that, if $x^m=y^n=u$ for some $x,y,u\in G_\infty$, where $m$ and $n$ are relatively prime, then $\langle x,y\rangle$ is cyclic. Let us show first that $\langle x,y\rangle$ is abelian. Since $\langle x,y\rangle$ is a nilpotent group, $\langle x,y\rangle/\langle u\rangle$ is nilpotent too. Moreover, $\langle x,y\rangle/\langle u\rangle=\langle \overline x,\overline y\rangle$, where $\overline x=x\langle u\rangle$ and $\overline y=y\langle u\rangle$. Since $o(\overline x)$ and $o(\overline y)$ are relatively prime, $\langle x,y\rangle/\langle u\rangle$ is cyclic. Let $\langle x,y\rangle/\langle u\rangle=\langle \overline z\rangle$, where $\overline z=z\langle u\rangle$ for some $z\in G$. Further, since $u\in\Z(\langle x,y\rangle)$, then $u$ and $z$ commute. Therefore, since $\langle x,y\rangle=\langle u,z\rangle$, we conclude that $\langle x,y\rangle$ is abelian.

 Let us show now that $\langle x,y\rangle$ is cyclic. By the Chinese remainder theorem, there are $k,l\in\mathbb Z$ such that $lm+kn=1$. Then $x,y\in\langle x^k y^l\rangle$ because
\begin{align*}
&(x^ky^l)^n=x^{kn}(y^n)^l=x^{kn}(x^m)^l=x^{kn+ml}=x\text{ and similarly}\\
&(x^ky^l)^m=y,
\end{align*}
and $(x^ky^l)^{mn}=u$.

Finally, let $P$ be a set of prime roots of $u$ such that, for every prime number $r$, $P$ contains at most one $r$-th root of $u$. Now, in order to prove the lemma, it is sufficient to show that $\langle P\rangle$ is a locally cyclic subgroup. Notice that, if $P_i$ denotes $\{x_1,x_2,\dots,x_i\}$, then by the above discussion, $\langle P_i\rangle$ is a cyclic subgroup of $\mathbf G$ for all $i$. Then
\[\langle P\rangle=\bigcup_{i\in\mathbb N}\langle P_i\rangle,\]
and $\langle P\rangle$ is a locally cyclic subgroup since $P_i\leq P_{i+1}$ for all $i$. Moreover, if $P$ was a maximal set of prime roots with the above-mentioned properties, then $\langle P\rangle$ is a maximal locally cyclic subgroup of $\mathbf G$. Thus, the lemma has been proven.
\end{proof}

Let us show now the main result of this section.

\begin{theorem}\label{Prufer groups are the only abelian groups with the property}
Let $\mathbf G$ and $\mathbf H$ be nilpotent groups none of which is a Pr\" ufer group. Then $\mathcal G(\mathbf G)\cong\mathcal G(\mathbf H)$ implies $\vec{\mathcal G}(\mathbf G)\cong\vec{\mathcal G}(\mathbf H)$.
\end{theorem}

\begin{proof}
Let $\mathbf G$ and $\mathbf H$ be nilpotent groups whose power graphs are isomorphic. By Theorem \ref{about all nilpotent groups}, since the group $\mathbf G$ is nilpotent, $\mathbf G_{<\infty}\leq\mathbf G$, and $\mathbf G_{<\infty}$ is the direct sum of its unique maximal $p$-subgroups. We separate this proof into two parts; in the first part, we assume that $\mathbf G_{<\infty}$ is non-quasicyclic, and, in the second part, we assume that  $\mathbf G_{<\infty}$ is quasicyclic. \vspace{2mm}

\noindent {\it Part 1.} Suppose now that $\mathbf G_{<\infty}$ is not a Pr\" ufer group. We remind the reader that a quasicyclic subgroup $\mathbf C_{p^\infty}$ is said to be intersection-free if each of its nonidentity elements is contained only in cyclic subgroups of $\mathbf G$ that are subgroups of $\mathbf C_{p^\infty}$. By \cite[Theorem 21]{zahirovic-2}, if $\mathbf G_{<\infty}$ contains no intersection-free quasicyclic subgroup, then $\vec{\mathcal G}(\mathbf G)\cong\vec{\mathcal G}(\mathbf H)$. Note that, if the nilpotent torsion group $\mathbf G_{<\infty}$ is not a $p$-group, then it does not have an intersection-free quasicyclic subgroup.

Therefore, the remaining case is when both $\mathbf G_{<\infty}$ and $\mathbf H_{<\infty}$ are non-quasi\-cyclic nilpotent $p$-groups whose some quasicyclic subgroups are intersection-free. If $\mathbf G_{<\infty}$ has a finite maximal cyclic subgroup of order $r^k$ for a prime number $r$, then both $\mathcal G(\mathbf G_{<\infty})$ and $\mathcal G(\mathbf H_{<\infty})$ have maximal cliques of order $r^k$, and thus, both $\mathbf G_{<\infty}$ and $\mathbf H_{<\infty}$ are $r $-groups. Therefore, $\big(\vec{\mathcal G}(\mathbf G)\big)[C_G]\cong\big(\vec{\mathcal G}(\mathbf H)\big)[C_H]$, where $\mathbf C_G$ and $\mathbf C_H$ are any pair of respective intersection-free quasicyclic subgroups of $\mathbf G$ and $\mathbf H$. Furthermore, maximal cliques of $\vec{\mathcal G}(\mathbf G_{<\infty})$ and $\vec{\mathcal G}(\mathbf H_{<\infty})$ and their intersections are quasicyclic or cyclic subgroups of $\mathbf G$ and $\mathbf H$, respectively. Also, since $\mathbf G_{<\infty}$ is a $p$-group, if $x,y\in G_{<\infty}$ have different closed neighborhoods in $\mathcal G(\mathbf G)$, then $x\rightarrow y$ if and only if $\overline N(x)\subset\overline N(y)$. Additionally, by \cite[Lemma 15]{zahirovic-2}, the closed neighborhood of an element $x$ in $\big(\mathcal G(\mathbf G)\big)[G_{<\infty}]$ is equal to
\[\bigcup_{i\in I} \mathrm{gen}(\mathbf C_{p^i}),\]
where $\mathrm{gen}(\mathbf C)$ denotes the set of all generators of a cyclic group $\mathbf C$, $I$ is some interval of natural numbers, and each $\mathbf C_{p^i}$ is some cyclic group of order $p^i$ so that $\mathbf C_{p^i}<\mathbf C_{p^{i+1}}$ for all $i$. Since the same holds in finite-order segment of $\mathcal G(\mathbf H)$, one can construct an isomorphism from $\vec{\mathcal G}(\mathbf G_{<\infty})$ to $\vec{\mathcal G}(\mathbf H_{<\infty})$ so that it maps elements each closed neighborhood of $\big(\mathcal G(\mathbf G)\big)[G_{<\infty}]$ onto elements of its respective closed neighborhood of $\big(\mathcal G(\mathbf H)\big)[H_{<\infty}]$. 
Now, by \cite[Theorem 10]{zahirovic-2}, it follows that $\vec{\mathcal G}(\mathbf G)\cong\vec{\mathcal G}(\mathbf H)$.

Finally, suppose that the nilpotent non-quasicyclic group $\mathbf G_{<\infty}$ has no finite maximal cyclic subgroup, and let $\mathbf G_{<\infty}$ be an $r$-group for a prime number $r$. Since $\mathbf G_{<\infty}$ is nilpotent, it has nontrivial center, and since it is not a Pr\" ufer group, it contains at least two quasicyclic subgroups $\mathbf C_1$ and $\mathbf C_2$. Since all quasicyclic subgroups of $\mathbf G_{<\infty}$ are intersection-free, $\mathbf C_1\cap\mathbf C_2$ is trivial, and, without loss of generality, we may suppose that $\Z(\mathbf G_{<\infty})\cap\mathbf C_1$ is nontrivial. Now, for $x\in C_1$ of order $r$ and $y\in C_2$ of order $r^2$, $\langle x,y\rangle$ is an abelian group, and $\langle y^2\rangle\leq \langle y\rangle,\langle xy\rangle$, which is a contradiction since $\mathbf C_2$ is intersection-free. Therefore, there is no such nilpotent group $\mathbf G_{<\infty}$.\vspace{2mm}

\noindent {\it Part 2.} Suppose now that $\mathbf G_{<\infty}$ is a Pr\" ufer group. Since $\mathbf G$ is not a Pr\" ufer group, $\mathbf G$ contains elements of infinite order, too. By \cite[Lemma 4]{zahirovic-2}, the finite-order segment of $\mathcal G(\mathbf H)$ is the infinite complete graph. Therefore, $\mathbf H_{<\infty}$ is also a Pr\" ufer group.

Let $p$ and $q$ be prime numbers so that $\mathbf G_{<\infty}$ and $\mathbf H_{<\infty}$ are the Pr\" ufer $p$-group and the Pr\" ufer $q$-group, respectively. Let $\varphi: G\rightarrow H$ be an isomorphism from $\mathcal G(\mathbf G)$ to $\mathcal G(\mathbf H)$. Let us prove that $p=q$. First we need to prove that $\varphi$ is an isomorphism between infinite-order segments of $\vec{\mathcal G}(\mathbf G)$ and $\vec{\mathcal G}(\mathbf H)$. Let $C\subseteq G$ induce a connected component of the infinite-order segment of $\mathcal G(\mathbf G)$, and let $x\in C$ be an element of $\mathbf G$ which has a $p$-th root. By Lemma \ref{numbers of prime roots}, $x$ has $p$ many $p$-th roots. Let $x_1$ and $x_2$ be two different $p$-th roots of $x$, i.e. $x=(x_1)^p=(x_2)^p$. Notice that, since $\overline N(x_1)\cap\overline N(x_2)\subseteq  O(x_1)\cap O(x_2)$, $\overline N(x_1)\cap\overline N(x_2)$ has only one nontrivial connected component and at most two isolated vertices. Then, by \cite[Proposition 7]{zahirovic-2}, $\varphi\rvert_{G_\infty}$ is an isomorphism from $\big(\vec{\mathcal G}(\mathbf G)\big)[C]$ to $\big(\vec{\mathcal G}(\mathbf H)\big)[\varphi(C)]$. Thus, $\varphi$ is an isomorphism between the infinite-order segments of $\vec{\mathcal G}(\mathbf G)$ and $\vec{\mathcal G}(\mathbf H)$.

It follows that $\varphi$ maps all maximal locally cyclic subgroups of $\mathbf G$ containing elements of infinite order onto all maximal locally cyclic subgroups of $\mathbf H$ containing elements of infinite order. Furthermore, notice that, for elements $x, y\in G$ of infinite order, $[y]_\equiv$ contains at most one prime root of $x$, and $[y]_\equiv$ contains a prime root of $x$ if and only if $y\rightarrow x$ and $x\rightarrow z\rightarrow y$ for no element $ z\not\in [x]_\equiv\cup[y]_\equiv$. Therefore, for any element $x\in G$ of infinite order, $\varphi$ maps the set of all $\equiv$-classes containing prime roots of $x$ onto the set of all $\equiv$-classes containing prime roots of $\varphi(x)$. Also, note that, if an element $y\in\mathbf G$ is contained in a subgroup of $\mathbf G$, then $[y]_\equiv$ is contained in that subgroup as well. Thus, $\varphi$ maps any minimal family of maximal locally cyclic subgroups of $\mathbf G$ whose union contains all prime roots of $x$ onto a minimal family of maximal locally cyclic subgroups of $\mathbf H$ whose union contains all prime roots of $\varphi(x)$. Therefore, by Lemmas \ref{p many p-th roots}, \ref{one q-th root} and \ref{numbers of prime roots}, it follows that $p=q$. Therefore, $\mathbf G_{<\infty}\cong\mathbf H_{<\infty}$. Hence, $\vec{\mathcal G}(\mathbf G)$ and $\vec{\mathcal G}(\mathbf H)$ have isomorphic finite-order segments. Also, by \cite[Theorem 10]{zahirovic-2}, $\vec{\mathcal G}(\mathbf G)$ and $\vec{\mathcal G}(\mathbf H)$ have isomorphic infinite-order segments, too. Therefore, $\vec{\mathcal G}(\mathbf G)\cong\vec{\mathcal G}(\mathbf H)$.
\end{proof}

\section{The sufficient condition is not a necessary one}

In the previous section, we showed that any two non-quasicyclic nilpotent groups with isomorphic power graphs also have isomorphic directed power graphs. It turns out that quasicyclic subgroups play an important role in describing what groups have the property that their power graphs determine their directed power graphs up to isomorphism. As it was proven in \cite{zahirovic-2}, if a group $\mathbf G$ does not contain a quasicyclic intersection-free subgroup, then for any group $\mathbf H$, if $\mathbf G$ and $\mathbf H$ have isomorphic power graphs, they also have isomorphic directed power graphs.

In this section, we find some groups $\mathbf G$ which contain a quasicyclic in\-ter\-se\-ction-free subgroup, but whose directed power graph does determine $\vec{\mathcal G}(\mathbf G)$ up to isomorphism. Moreover, the power graph of any of these groups not only determines the directed power graph of the group, but it determines the group itself up to isomorphism. Thus, the sufficient condition from \cite{zahirovic-2} is not a necessary one as well. Before introducing the main result of this section, we are going to prove several Lemmas first. We remind the reader that $\pi_i$ denotes the projection map. 

\begin{lemma}\label{all maximal locally cyclic are like Q}
Let $\mathbf A$ be a maximal locally cyclic torsion-free subgroup of $\mathbf H=\mathbf H_1\times\mathbf C_{p^\infty}$, where $\mathbf H_1\cong\mathbb Q$ and $p$ is a prime number. Then $\pi_1(A)=H_1$. Moreover, $\mathbf A\cong\mathbb Q$ and $\mathbf H=\mathbf A\times\mathbf C_{p^\infty}$.
\end{lemma}

\begin{proof}
Let us prove first that $\mathbf A$ is a divisible group. Suppose that some element $u\in A$ does not have a $q$-th root in $\mathbf A$ for some prime number $q$. Since $\mathbf A$ is locally cyclic, then $\mathbf A=\bigcup_{i\in\mathbb N} \langle u_i\rangle$ for some elements $u_i\in A$ where $u_1=u$ and $\langle u_i\rangle<\langle u_{i+1}\rangle$ for every $i\in\mathbb N$. Since $\mathbf H$ is a divisible group, every $u_i$ has a $q$-th root $v_i$ such that $\langle v_i\rangle<\langle v_{i+1}\rangle$. Namely, if $q\neq p$, then every $u_i$ has the unique $q$-th root $v_i$. On the other hand, if $q=p$, then every $u_i$ has $p$-many $q$-th roots; in this case we choose $v_1$ to be any $p$-th root of $u_1$, and for each $i\in\mathbb N$, we choose the one $p$-th root $v_{i+1}$ such that $v_i\in\langle v_{i+1}\rangle$. This way, we obtain a locally cyclic subgroup $\mathbf B=\bigcup_{i\in \mathbb N}\langle v_i\rangle$ such that $\mathbf A\lneq\mathbf B$, which is a contradiction. Therefore, group $\mathbf A$ is divisible.

Let $K=\pi_1(A)$. Notice that $\mathbf K$ is a nontrivial subgroup of $\mathbf H_1$ because, otherwise, $\mathbf A$ would be a torsion group. Notice that the group $\mathbf K$ divisible since it is a homomorphic image of the divisible group $\mathbf A$. Therefore, since $\mathbf H_1\cong\mathbb Q$ and the group of rational numbers has no proper divisible subgroup, it follows that $K=H_1$. Thus, $\mathbf K\cong\mathbb Q$ because $\mathbb Q$ is the only divisible locally cyclic torsion-free group up to isomorphism.

Let us show now that $\mathbf A\cong\mathbb Q$. Since $\pi_1(A)=H_1$, $\pi_1$ is a surjection from $A$ to $H_1$. Furthermore, $\pi_1$ is also an injection because $\mathbf A$ is a locally cyclic subgroup of $\mathbf H$. Therefore, $\pi_1$ is an isomorphism from $\mathbf A$ to $\mathbf H_1$, and thus, $\mathbf A\cong\mathbb Q$. Moreover, $\mathbf A\cap\mathbf C_{p^\infty}$ is the trivial group, $\mathbf A\mathbf C_{p^\infty}=\mathbf H$, and every element of $\mathbf A$ commutes with each element of $\mathbf C_{p^\infty}$. Thus, $\mathbf H=\mathbf A\times\mathbf C_{p^\infty}$.
\end{proof}

\begin{lemma}\label{maximal locally cyclic onto maximal locally cyclic}
Let $\mathbf H=\mathbf H_1\times\mathbf C_{p^\infty}$, where $\mathbf H_1\cong \mathbb Q$ and $p$ is a prime number, and let $\mathbf G$ be a group. Let $\varphi$ be an isomorphism from $\mathcal G(\mathbf H)$ to $\mathcal G(\mathbf G)$. If $\mathbf A$ is a maximal locally cyclic torsion-free subgroup of $\mathbf H$, then $\mathbf B=\varphi(\mathbf A)$ is a maximal locally cyclic torsion-free subgroup of $\mathbf G$ isomorphic to $\mathbb Q$.
\end{lemma}

\begin{proof}
Let $x\in H$. Every element of infinite order of $\mathbf H$ has $p$ many $p$-th roots. Let $x_1$ and $x_2$ be two different $p$-th roots of $x$. Then the subgraph of $\overline{\mathcal G(\mathbf H)}$ induced by $\overline N(x_1)\cap\overline N(x_2)$ has only one nontrivial connected component and only two isolated vertices since 
\[\overline N(x_1)\cap\overline N(x_2)= O(x_1)\cap O(x_2)=\langle x\rangle\setminus\{e\}.\]
 Thus, by \cite[Proposition 7]{zahirovic-2}, it follows that $\varphi|_{G_\infty}$ is an isomorphism between the infinite-order segments of $\vec{\mathcal G}(\mathbf H)$ and $\vec{\mathcal G}(\mathbf G)$. Therefore, $B=\varphi(A)$ is the universe of a locally cyclic subgroup of $\mathbf G$. Furthermore, if $\mathbf B$ was not a maximal locally cyclic subgroup, then there would be some locally cyclic subgroup $\mathbf D$ of $\mathbf G$ such that $\mathbf B<\mathbf D$. Then $\varphi^{-1}(D)$ would be the universe a locally cyclic subgroup $\mathbf C$ of $\mathbf H$ such that $\mathbf A<\mathbf C$, which would be a contradiction. Therefore, $\varphi$ maps the maximal locally cyclic torsion-free subgroup $\mathbf A$ of $\mathbf H$ onto the maximal locally cyclic torsion-free subgroup $\mathbf B$ of $\mathbf G$. Furthermore, by \cite[Corollary 2]{zahirovic-1}, $\mathbf B$ is isomorphic to $\mathbb Q$.
\end{proof}

\begin{lemma}\label{torsion part is a central subgroup}
Let $\mathbf H=\mathbf H_1\times\mathbf C_{p^\infty}$, where $\mathbf H_1\cong \mathbb Q$ and $p$ is a prime number, and let $\mathbf G$ be a group such that $\mathcal G(\mathbf G)\cong\mathcal G(\mathbf H)$. Then $\mathbf G_{<\infty}$ is a central subgroup of $\mathbf G$.
\end{lemma}

\begin{proof}
Since $H_{<\infty}$ induces a complete subgraph of $\mathcal G(\mathbf H)$, the finite-order segment $\mathcal G(\mathbf G)$ is also a complete graph. Therefore, $\mathbf G_{<\infty}\cong\mathbf C_{q^\infty}$ for some prime number $q$. Moreover, $\mathbf G_{<\infty}$ is a normal subgroup of $\mathbf G$. Let $a\in G_\infty$. Then there is a maximal locally cyclic subgroup $\mathbf A$ of $\mathbf H$ which contains $\varphi^{-1}(a)$. Now, by Lemma \ref{maximal locally cyclic onto maximal locally cyclic}, $\mathbf B=\varphi(\mathbf A)$ is a maximal locally cyclic subgroup of $\mathbf G$. Furthermore, by Lemmas \ref{all maximal locally cyclic are like Q} and \ref{maximal locally cyclic onto maximal locally cyclic}, it follows that $\mathbf B\cong\mathbb Q$. Since $\mathbf G_{<\infty}\unlhd \mathbf G$, we have that $\mathbf L=\mathbf B\mathbf G_{<\infty}\leq\mathbf G$. Let $x$ be an element of $\mathbf G_{<\infty}$, and let $o(x)=q^n$. By $N/C$ Lemma (see \cite[Lemma 7.1]{rotman}), $\mathrm N_{\mathbf L}(\langle x\rangle)/\mathrm C_{\mathbf L}(\langle x\rangle)$ can be embedded into $\Aut(\langle x\rangle)$. Notice that $\Aut(\langle x\rangle)$ is finite and $\mathrm N_{\mathbf L}(\langle x\rangle)=\mathbf L$. Therefore,
\[[\mathbf B:\mathbf B\cap\mathrm C_{\mathbf L}(\langle x\rangle)]\leq [\mathbf L:\mathrm C_{\mathbf L}(\langle x\rangle)]<\infty.\]
However, $\mathbb Q$ has no subgroup of finite index. Thus, $\mathbf B\cap \mathrm C_{\mathbf L}(\langle x\rangle)=\mathbf B$, which implies that $\mathbf B\leq\mathrm C_{\mathbf L}(\langle x\rangle)$. Therefore, $a\in\mathrm C_{\mathbf G}(x)$ for every $x\in\mathbb G_{<\infty}$. It follows that $\mathbf G_{<\infty}\leq\mathrm Z(\mathbf G)$. This proves the lemma.
\end{proof}

Now, we shall prove the main result of this section.

\begin{theorem}
Let $\mathbf G$ be a group, and let $\mathcal G(\mathbb Q\times\mathbf C_{p^\infty})\cong\mathcal G(\mathbf G)$, where $p$ is a prime number. Then $\mathbf G\cong\mathbb Q\times \mathbf C_{p^\infty}$. 
\end{theorem}

\begin{proof}
In order to avoid any ambiguity and have all the groups with multiplicative notation, let us prove that $\mathcal G(\mathbf H_1\times\mathbf C_{p^\infty})\cong\mathcal G(\mathbf G)$ if and only if $\mathbf G\cong\mathbf H_1\times \mathbf C_{p^\infty}$ for $\mathbf H_1\cong\mathbb Q$. Also, let $\mathbf H$ denote $\mathbf H_1\times\mathbf C_{p^\infty}$.

Let $\varphi:H\rightarrow G$ be an isomorphism from $\mathcal G(\mathbf H)$ to $\mathcal G(\mathbf G)$. By \cite[Lemma 4]{zahirovic-2}, $\varphi$ maps the finite-order segment of $\mathcal G(\mathbf H)$ onto the finite-order segment of $\mathcal G(\mathbf G)$. Therefore, the subgraph of $\mathcal G(\mathbf G)$ induced by $G_{<\infty}$ is complete, and thus, $\mathbf G_{<\infty}\cong\mathbf C_{q^\infty}$ for a prime number $q$. 

Let us show now that $\varphi\rvert_{H_\infty}$ is an isomorphism between the infinite-order segments of $\vec{\mathcal G}(\mathbf H)$ and $\vec{\mathcal G}(\mathbf G)$. Notice that there are elements $x, x_1, x_2\in H_{\infty}$, $x_1\neq x_2$, such that $(x_1)^p=(x_2)^p=x$. Then $\overline N(x_1)\cap\overline N(x_2)=  O(x_1)\cap O(x_2)=\langle x\rangle\setminus\{e\}$. Therefore, the subgraph of $\overline{\mathcal G(\mathbf H)}$ induced by $\overline N(x_1)\cap\overline N(x_2)$ has only one nontrivial connected component, and $x$ and $x^{-1}$ are its only isolated vertices. Thus, by \cite[Proposition 7]{zahirovic-2}, $\varphi\rvert_{H_\infty}$ is an isomorphism between the infinite-order segments of the directed power graphs of $\mathbf H$ and $\mathbf G$. 
 It follows that $\varphi$ maps maximal locally cyclic subgroups of $\mathbf H$ onto maximal locally cyclic subgroups of $\mathbf G$.

For all $x,y\in H_\infty$, since $\varphi\rvert_{H_\infty}$ is an isomorphism between $\vec{\mathcal G}(\mathbf H)[H_\infty]$ and $\vec{\mathcal G}(\mathbf G)[G_\infty]$, as discussed in the proof of Theorem \ref{Prufer groups are the only abelian groups with the property}, $[y]_\equiv$ contains a prime root of $x$ if and only if $[\varphi(y)]_\equiv$ contains a prime root of $\varphi(x)$. We will say that $y$ is a common prime root of $x$ if $y$ belongs to all maximal locally cyclic subgroups containing $x$. Since $\varphi$ maps the set of all maximal locally cyclic subgroups of $\mathbf H$ onto the set of all maximal cyclic subgroups of $\mathbf G$, $[y]_\equiv$ contains a common prime root of $x$ if and only if $[\varphi(y)]_\equiv$ contains a common prime root of $\varphi(x)$ for all $x,y\in H_\infty$. By Lemma \ref{all maximal locally cyclic are like Q}, for each prime number $r$, any maximal locally cyclic subgroup of $\mathbf H$ contains an $r$-th root of each of its elements. Furthermore, notice that any prime root of an infinite-order element $x$ of $\mathbf H$ is common if and only if it is not a $p$-th root of $x$. In the remainder of this proof, we will show that $p=q$.

Let $x\in H_\infty$. Then $x$ has $p$ many noncommon prime roots. Moreover, for every noncommon prime root $y$ of $x$, we have that $y\rightarrow x\rightarrow x^p$ and there is no element $z\not\in[y]_\equiv\cup[x]_\equiv\cup[x^p]_\equiv$ such that $y\rightarrow z\rightarrow x^p$. Since $\vec{\mathcal G}(\mathbf H)[H_\infty]\cong\vec{\mathcal G}(\mathbf G)[G_\infty]$ and every $\equiv$-class contains at most one prime root of $x$, $\varphi(x)$ also has $p$ many noncommon prime roots. Notice that, by Lemma \ref{torsion part is a central subgroup}, all $q$-th roots of $\varphi(x)$ are noncommon roots of $\varphi(x)$ in $\mathbf G$. Therefore, $\varphi(x^p)\in[(\varphi(x))^q]_\equiv$. Furthermore, if $\varphi(x)$ had any noncommon prime root which is an $r$-th root $y$, $r\neq q$, then we would have $y\rightarrow y^{q}\rightarrow (\varphi(x))^q$ and $y^q\not\in [y]_\equiv\cup[\varphi(x)]_\equiv\cup[(\varphi(x))^q]_\equiv$, which would be a contradiction because $\varphi$ is an isomorphism between $\vec{\mathcal G}(\mathbf H)[H_\infty]$ and $\vec{\mathcal G}(\mathbf G)[G_\infty]$. Thus, all noncommon prime roots of $\varphi(x)$ are exactly all $q$-th roots.

By Lemma \ref{torsion part is a central subgroup}, $\mathbf G_{<\infty}\cong\mathbf C_{q^\infty}$ is a central subgroup of $\mathbf G$. Let $\mathbf C_q$ be the unique subgroup of $\mathbf G$ of order $q$. Then, if $u$ is a $q$-th root of $\varphi(x)$, so are all elements of the coset $uC_p$. Therefore, the number of all noncomon prime roots of $\varphi(x)$ is $kq$ for some $k\in\mathbf N$. Since $p=kq$, and since $p$ is a prime number, we have obtained that $p=q$.

It remains for us to show that $\mathbf G$ and $\mathbf H$ are isomorphic. Let $\mathbf A$ be a maximal torsion-free locally cyclic subgroup of $\mathbf G$. Then $\varphi^{-1}(\mathbf A)$ is a maximal torsion-free locally cyclic subgroup of $\mathbf G$, which by Lemma \ref{all maximal locally cyclic are like Q}, implies that $\varphi^{-1}(\mathbf A)\cong\mathbb Q$. Therefore, by Lemma \ref{maximal locally cyclic onto maximal locally cyclic}, $\mathbf A\cong\mathbb Q$. Let $\mathbf T=\mathbf A\mathbf G_{<\infty}$. Since $\mathbf T\cong\mathbb Q\times\mathbf C_{p^\infty}$ contains all roots of each of its elements, and since the infinite-order segment of $\mathcal G(\mathbf G)$ is a connected graph, it follows that $T_\infty=G_\infty$. Therefore, $\mathbf G=\mathbf T\cong\mathbf H$. This completes our proof.
%
\end{proof}

%

\vbox{
\vbox{\noindent
Sayyed Heidar Jafari

Shahrood University of Technology, Faculty of Mathematical Sciences, Iran

{\it e-mail}: \href{mailto:shjafari55@gmail.com}{shjafari55@gmail.com}}\

\vbox{\noindent
Samir Zahirovi\' c

University of Novi Sad, Department of Mathematics and Informatics, Serbia

{\it e-mail}: \href{mailto:samir.zahirovic@dmi.uns.ac.rs}{samir.zahirovic@dmi.uns.ac.rs}}}

\end{document}